\documentclass{amsart}
\usepackage{amsmath,amsthm,amssymb}

\makeatletter
    
    \@addtoreset{equation}{section}
  \makeatother

\newtheorem{definition}{Definition}[section]

\newtheorem{theorem}[definition]{Theorem}
\newtheorem{lemma}[definition]{Lemma}
\newtheorem{corollary}[definition]{Corollary}
\newtheorem{remark}{Remark}[section]

\newcommand{\re}{\mathbb R} 

\DeclareMathOperator{\diver}{div}
\DeclareMathOperator{\rot}{rot}

\pagebreak

\title[2D Navier-Stokes equations]{
Asymptotic properties of steady and nonsteady solutions
to the 2D Navier-Stokes equations with finite generalized Dirichlet integral
}
\author{Hideo Kozono, Yutaka Terasawa and Yuta Wakasugi}

\address[H. Kozono]{Department of Mathematics, Faculty of Science and Engineering,
Waseda University, Tokyo 169--8555, Japan, 
Research Alliance Center of Mathematical Sciences, Tohoku University, 
Sendai 980-8578, Japan}
\email[H. Kozono]{kozono@waseda.jp, hideokozono@tohoku.ac.jp}
\address[Y. Terasawa]{Graduate School of Mathematics, Nagoya University,
Furocho Chikusaku Nagoya 464-8602, Japan}
\email[Y. Terasawa]{yutaka@math.nagoya-u.ac.jp}
\address[Y. Wakasugi]{Graduate School of Engineering,
Hiroshima University,
Higashi-Hiroshima, 739-8527, Japan
}
\email[Y. Wakasugi]{wakasugi@hiroshima-u.ac.jp}
\begin{document}
\begin{abstract}
We consider the stationary and non-stationary Navier-Stokes equations in the whole plane $\re^2$ 
and in the exterior domain outside of the large circle.  
The solution $v$ is handled in the class with $\nabla v \in L^q$ for $q \ge 2$. 
Since we deal with the case $q \ge 2$,
our class is larger in the sense of spatial decay at infinity than that of the finite Dirichlet integral,
i.e., for $q=2$ where a number of results such as asymptotic behavior of solutions have been observed.  
For the stationary problem we shall show that
$\omega(x)= o(|x|^{-(1/q + 1/q^2)})$
and
$\nabla v(x) = o(|x|^{-(1/q+1/q^2)} \log |x|)$
as $|x| \to \infty$, where $\omega \equiv \rot v$.  
As an application,
we prove the Liouville type theorem under the assumption that 
$\omega \in L^q(\re^2)$.
For the non-stationary problem, a generalized $L^q$-energy identity is clarified.     
We also apply it to 
the uniqueness of the Cauchy problem and the Liouville type theorem for ancient solutions
under the assumption that
$\omega \in L^q(\re^2 \times I)$.
\end{abstract}

\maketitle
\section{Introduction}
\footnote[0]{2010 Mathematics Subject Classification. 35Q30; 35B53; 76D05}

We consider the homogeneous stationary Navier-Stokes equations
\begin{align}%
\label{sns}
	\left\{ \begin{array}{l}
		- \Delta v + ( v \cdot \nabla ) v + \nabla p = 0, \\
		\diver v = 0
	\end{array}\right.
\end{align}%
in
$\mathbb{R}^2$ or an exterior domain 
$B_{r_0}^c \equiv \{ x=(x_1, x_2) \in \mathbb{R}^2 ; r = |x| > r_0 \}$
with some constant
$r_0 > 0$.
Here,
$v = v(x) = (v_1(x), v_2(x))$
and
$p=p(x)$
are the velocity vector and the scalar pressure, respectively.
\par
In the pioneer paper by Leray \cite{Le}, there is a solution $v$ of (\ref{sns}) 
with 
\begin{equation}\label{FDI}
\int_{r > r_0}|\nabla v(x)|^2 dx < \infty
\end{equation}
satisfying the homogeneous boundary condition on $r=r_0$.  
It has been an open question whether,  under the condition (\ref{FDI}),  
$v$ behaves like 
\begin{equation}\label{asym0}
v(x) \to v_{\infty}
\quad
\mbox{as $|x| \to \infty$}
\end{equation}
with some constant vector $v_{\infty} \in \re^2$.  
Gilbarg--Weinberger \cite{GiWe78} proved that if the solution $v$ in the class (\ref{FDI}) satisfies 
\begin{equation}\label{bounded} 
v\in L^{\infty}(B_{r_0}^c)
\end{equation} 
then there is constant vector $v_{\infty} \in \re^2$ such that 
\begin{equation}\label{mean}
\lim_{r \to\infty}\int_0^{2\pi}|v(r, \theta) - v_{\infty}|d\theta = 0
\end{equation}
with 
\begin{align}\label{asym1}
	&\omega (r, \theta) = o(r^{-3/4}),\\
	&\nabla v (r,\theta) = o(r^{-3/4} (\log r ))
\end{align} 
uniformly in $\theta\in (0, 2\pi)$ as $r \to\infty$,
where
$\omega$
is the vorticity
$\omega= \partial_{x_1} v_2- \partial_{x_2}v_{1}$.
Later on Amick \cite{Am88} and Korobkov--Pileckas--Russo \cite{KoPiRu17} 
proved that every solution $v$ of (\ref{sns}) 
with the finite Dirichlet integral (\ref{FDI}) 
is bounded as in (\ref{bounded}), and so necessarily satisfies (\ref{mean}) and (\ref{asym1}). 
Recently, Korobkov--Pileckas--Russo \cite{KoPiRu18} succeeded to obtain a remarkable result 
which states that every solution $v$ of (\ref{sns}) with (\ref{FDI}) converges uniformly at infinity, 
i.e., 
\begin{equation}\label{uniform} 
\lim_{r\to\infty}\sup_{\theta\in (0, 2\pi)}|v(r, \theta)- v_\infty|=0, 
\end{equation}
where $v_{\infty}\in \re^2$ is the constant vector as in (\ref{mean}).  
On the other hand, for the small prescribed constant vector $v_\infty \in \re^2$, 
the existence of solutions $v$ of (\ref{sns}) with (\ref{FDI}) 
having the parabolic wake region has been fully investigated by Finn-Smith \cite{FiSm67}.  
\par
In this paper, we consider a generalized class 
\begin{align}%
\label{nablavlq}
\int_{r > r_0} |\nabla v(r,\theta)|^q \,dx < \infty
\quad
\mbox{with some $2 < q < \infty$.}
\end{align}%
Since our domain is unbounded and since we are interested in the asymptotic behavior of solutions 
$v$ of (\ref{sns}), the larger $q$, the weaker the assumption on the decay of $\nabla v$
at the spatial infinity. 
On the other hand, so far as the problem on existence and uniqueness of solutions to (\ref{sns}), 
the class of solutions $v$ is restricted to (\ref{FDI}).  
Even for the linearized Stokes equations, it is difficult to find solutions $v$ satisfying (\ref{nablavlq}) 
but not (\ref{FDI}). For instance, it is shown by Kozono-Sohr \cite{KoSo92} that 
the Stokes equations are uniquely solvable if and only if $v$ satisfies  (\ref{FDI}).   
In this respect, our motivation might be artificial. 
However, from the viewpoint of the relation between the explicit decay rate at the spacial infinity of solutions 
and their class such as (\ref{nablavlq}), we bring focus onto the effect of the value of general $2 < q < \infty$. 
Furthermore, it seems also an interesting question how large we may take $q$ in (\ref{nablavlq}) 
for validity of the Liouville-type theorem.      
\par
\bigskip
Our first result is on the following
pointwise decay estimates for
the vorticity $\omega= \partial_{x_1} v_2- \partial_{x_2}v_{1}$
and
$\nabla v$.
\begin{theorem}\label{thm1}
Let $v$
be a smooth solution of \eqref{sns} in
$B_{r_0}^c=\{ x \in \mathbb{R}^2 ; r > r_0 \}$
satisfying \eqref{nablavlq} with some
$q \in (2,\infty)$.
Then, we have
\begin{align}%
\label{estomegav}
	|\omega (r,\theta)| &= o (r^{-(1/q+1/q^2)}) \quad (r\to \infty), \\
\label{estnablav}
	|\nabla v (r,\theta)| &= o (r^{-(1/q+1/q^2)} \log r) \quad (r\to \infty)
\end{align}%
uniformly in
$\theta \in [0,2\pi]$.
\end{theorem}

\begin{remark}
From the assumption \eqref{nablavlq} with a density argument,
we can see that
the velocity $v$ itself satisfies
\begin{align}
\label{estv}
	|v (r,\theta)| &= o(r^{1-2/q}) \quad (r \to \infty)
\end{align}
(see Lemma \ref{lem_lq1}).
\end{remark}

Theorem \ref{thm1} and the maximum principle immediately give
the following Liouville type theorem.
\begin{corollary}[Liouville type theorem]\label{cor_liouville1}
Let $v$ be a smooth solution of
\eqref{sns} in $\mathbb{R}^2$ satisfying
\begin{align*}%
	\nabla v \in L^q(\mathbb{R}^2)\quad
	\left( \mbox{or}\  \omega \in L^q(\mathbb{R}^2) \right)
\end{align*}%
with some $q \in (2,\infty)$.
Then, $v$ must be a constant vector.
\end{corollary}

\begin{remark}
{\rm (i)}
When $q=2$, the corresponding result was given by
Gilbarg--Weinberger \cite[Theorem 2]{GiWe78}.

\noindent
{\rm (ii)}
Bildhauer--Fuchs--Zhang \cite[Theorem 1.5]{BiFuZh13}
proved the Liouville type theorem in
$\mathbb{R}^2$
under the assumption
$\limsup_{|x|\to \infty} |x|^{-\alpha} |v(x)| < \infty$
with some
$\alpha < 1/3$.
Thus, noting \eqref{estv}, we see that 
Corollary \ref{cor_liouville1} is included in \cite{BiFuZh13} if
$2<q< 3$.
Therefore, our novelty is the case
$3 \le q <\infty$.

\noindent
{\rm (iii)}
When
$q=\infty$,
the assertion of the above corollary is not true.
Indeed,
$v = (x_1, -x_2)$
and
$p= -\frac12 |x|^2$
satisfy \eqref{sns} and $\omega = 0$,
while
$v$
is not a constant vector
(this example is given in
\cite[Remark 3.1]{LeZhZh17}).

\noindent
{\rm (iv)}
Recently, the Liouville-type theorem for 3D stationary Navier-Stokes equations
is intensively studied.
We refer the reader to
\cite{Ch1, Ch15, ChWo, Ga, KoTeWa17, Se}.

\noindent
{\rm (v)}
In the appendix, we give an alternate proof of
Corollary \ref{cor_liouville1}
which does not rely on the estimates established in Theorem \ref{thm1}.
\end{remark}

Next, we consider the 2D non-stationary Navier-Stokes equations
\begin{align}%
\label{ns}
	\left\{ \begin{array}{ll}
		\partial_t v - \Delta v + ( v \cdot \nabla ) v + \nabla p = 0,&x \in \mathbb{R}^2, t\in I,\\
		\diver v = 0, &x \in \mathbb{R}^2, t\in I,\\
	\end{array}\right.
\end{align}%
where $I \subset \re$ is an interval.
We also deal with the corresponding problem for the vorticity
\begin{align}%
\label{ns_vor}
		\partial_t \omega - \Delta \omega + v \cdot \nabla \omega = 0,
		\quad x \in \mathbb{R}^2, t\in I.
\end{align}%

\begin{theorem}[$L^q$-energy identity]\label{thm_energy_id}
Let
$\omega$
be a smooth solution of \eqref{ns_vor}.
We assume that there exists
$q \in [2, \infty)$
such that
\begin{equation}\label{eqn:1.13}
	\omega \in L^q(\mathbb{R}^2 \times I)
\end{equation}%
holds.
Then, for any $t, t' \in I$ with $t ' < t$,
we have the $L^q$-energy identity
\begin{align}
	& \int_{\re^2} |\omega(x,t)|^{q}\,dx
	+ q(q-1) \int_{t'}^t \int_{\re^2} |\nabla \omega(x,\tau)|^2 |\omega(x,\tau)|^{q-2} \,dx d\tau \label{eqn:1.14} \\
	&= \int_{\re^2} |\omega(x,t')|^{q}\,dx, \nonumber
\end{align}
provided that
$\omega(\cdot,t') \in L^q(\mathbb{R}^2)$.
\end{theorem}

From the above theorem, we immediately have
the uniqueness of the solution to the Cauchy problem,
and the Liouville-type theorem for the ancient solution.
\begin{corollary}\label{cor_liouville}
{\rm (i) (Uniqueness of the Cauchy problem)}
Let $T>0$ and $I = [0,T]$, and let
$\omega$ be a smooth solution to \eqref{ns_vor} satisfying \eqref{eqn:1.13}
with some $q \in [2,\infty)$.
Moreover, we assume
$\omega (x,0) = 0$.
Then, we have $\omega \equiv 0$.\\
{\rm (ii) (Liouville-type theorem for ancient solutions)}
Let $I = (-\infty, 0)$
and $\omega$ an ancient solution to \eqref{ns_vor} satisfying
 \eqref{eqn:1.13}
with some $q \in [2,\infty)$.
Then, we have $\omega \equiv 0$.
\end{corollary}
\begin{remark}
The identity (\ref{eqn:1.14}) is invariant under the scaling
$\omega(x,t) \mapsto \lambda^2 \omega(\lambda x, \lambda^2 t)$.
It is known that in the 2D case, every weak solution $v$ of (\ref{ns}) in the Leray-Hopf class, i.e., 
$v \in L^\infty(0, T; L^2(\re^2))\cap L^2(0, T; H^1(\re^2))$ is smooth and satisfies 
the energy identity
\begin{align*}
\int_{\re^2}|v(x, t)|^2 dx + 2\int_0^t\int_{\re^2}|\nabla v(x, \tau)|^2d\tau 
= \int_{\re^2}|v_0(x)|^2dx
\end{align*}
for all $t \in (0, T)$.  Hence, we may regard Theorem \ref{thm_energy_id} as the 
result on the 
generalized $L^q$-energy identity for the vorticity $\omega$.  
\end{remark}

\section{Proof of Theorem \ref{thm1}}
\subsection{Proof of the estimate of the vorticity and Corollary \ref{cor_liouville1}}
We first prove \eqref{estomegav}.
Hereafter, we denote by
$C = C(\ast, \cdots, \ast)$
positive constants depending only on the quantities appearing in parentheses.

The proof of Theorem \ref{thm1} is similar to that of
\cite[Theorem 5]{GiWe78}.
Our idea is to control the quantity
\begin{align*}
	\int |\nabla \omega|^2 |\omega|^{q-2}\,dx.
\end{align*}
A similar quantity is also used in \cite{BiFuZh13}.

We first show the asymptotic behavior for a vector field
satisfying \eqref{nablavlq}.

\begin{lemma}\label{lem_lq1}
If a vector field
$v = v(x) = (v_1, v_2) \in L^q_{{\rm loc}}(B_{r_0}^c)$
satisfies \eqref{nablavlq} with some
$q \in (2, \infty)$,
then
\begin{align}
\label{vest}
	\lim_{r\to \infty} r^{-(1-2/q)} \sup_{\theta \in(0, 2\pi)} |v(r,\theta)| = 0.
\end{align}
\end{lemma}
\begin{proof}
We first take a cut-off function
$\chi \in C^{\infty}(\mathbb{R}^2)$
so that
$\chi(x) = 0$ $(|x| \le r_0)$
and
$\chi(x) = 1$ $(|x| \ge r_0+1)$.
Then, we have
$\nabla (\chi v) \in L^q(\mathbb{R}^2)$.
Therefore, by \cite[Lemma 2.1]{KoSo92}, for any
$\varepsilon >0$,
there exists
$v_{\varepsilon} \in C_0^{\infty}(\mathbb{R}^2)$
satisfying
$\| \nabla(\chi v) - \nabla v_{\varepsilon} \|_{L^q(\mathbb{R}^2)} < \varepsilon$.
Let
$w_{\varepsilon} = \chi v - v_{\varepsilon}$.
Then, by the Sobolev embedding, we estimate
\begin{align*}
	\limsup_{r \to \infty} r^{-(1-2/q)} \sup_{\theta \in (0, 2\pi)} |v(r,\theta)|
		&\le \limsup_{r \to \infty} r^{-(1-2/q)} \sup_{\theta \in (0, 2\pi)}
			|w_{\varepsilon}(r,\theta) - w_{\varepsilon}(0)| \\
		&\quad + \limsup_{r \to \infty} r^{-(1-2/q)} |w_{\varepsilon}(0)|\\
		&\quad + \limsup_{r \to \infty} r^{-(1-2/q)}
			\sup_{\theta \in(0, 2\pi)} |v_{\varepsilon}(r,\theta)|\\
		&\le C \| \nabla w_{\varepsilon} \|_{L^q(\mathbb{R}^2)} \\
		&\le C \varepsilon,
\end{align*}
where the constant
$C>0$
is independent of
$\varepsilon$.
Since $\varepsilon$ is arbitrary, we conclude
$\limsup_{r \to \infty} r^{-(1-2/q)} \sup_{\theta \in (0, 2\pi)} |v(r,\theta)| = 0$.
\end{proof}

\begin{lemma}\label{lem_lq2}
Let $r_0 >0$. Suppose that
$v=(v_1, v_2)$ is a smooth solution of \eqref{sns} in $r>r_0$.  
If there exists
$q \in (2,\infty)$
such that
\begin{align*}
	\int_{r>r_0} |\nabla v|^q\,dx < \infty, 
\end{align*}
then the vorticity
$\omega = \partial_{x_1} v_2 - \partial_{x_2} v_1$
satisfies
\begin{align*}
	\int_{r>r_1} r^{2/q} |\nabla \omega|^2 |\omega|^{q-2}\,dx < \infty
\end{align*}
for all $r_1 > r_0$.
\end{lemma}

\begin{proof}
Let
$R>r_1>r_0$
and take
$\xi_1, \xi_2 \in C^{\infty}((0,\infty))$
so that
\begin{align*}%
	\xi_1(r) = \begin{cases}
		0 &(r \le \frac{r_0+r_1}{2}),\\
		1 &(r \ge r_1),
		\end{cases}
		\quad
		\xi_2(r) = \begin{cases}
			1& (r \le 1),\\
			0& (r\ge 2).
		\end{cases}
\end{align*}%
We define
\begin{align*}%
	\eta_R (r) = r^{2/q} \xi_1(r) \xi_2 \left( \frac{r}{R} \right).
\end{align*}%
We easily see that
\begin{align*}
	| \nabla \eta_R| \le Cr^{-1+2/q},\quad  |\Delta \eta_R| \le C 
\end{align*}
hold with some constant $C>0$.

Let
$h=h(\omega)$
be a $C^1$ and piecewise $C^2$ function, which is determined later.
Then,
noting $\nabla (h(\omega(x)) = h'(\omega(x)) \nabla \omega(x)$
and
using the condition
$\diver v = 0$,
we have the following identity
\cite[p. 385]{GiWe78}
\begin{align*}%
	\diver\left[
		\eta_R \nabla h - h \nabla \eta_R - \eta_R h v \right]
		&= \eta_R h'' |\nabla \omega|^2
			-h \left[ \Delta \eta_R + v \cdot \nabla \eta_R \right]
			+ \eta_R h' \left[ \Delta \omega - v \cdot \nabla \omega \right].
\end{align*}%
Since
$\omega$
satisfies the vorticity equation
$\Delta \omega - v \cdot \nabla \omega = 0$
and
$\eta_R = 0$
near
$r=r_0$ and $r=\infty$,
integrating the above identity on
$\{ r > r_0 \}$,
we have
\begin{align}%
\label{eq_h_en}
	\int_{r>r_0} \eta_R h''(\omega) |\nabla \omega|^2 \,dx
		= \int_{r>r_0} h(\omega) ( \Delta \eta_R + v\cdot \nabla \eta_R )\,dx.
\end{align}%
Taking
$h(\omega) = |\omega|^q$,
we obtain
\begin{align*}
	\int_{r_1 < r < R} r^{2/q}
		|\nabla \omega|^2 |\omega|^{q-2}\, dx
	&\le C \int_{r>r_0} |\omega|^q |\Delta \eta_R + v\cdot \nabla \eta_R |\,dx.
\end{align*}
From Lemma \ref{lem_lq1}, we see
$r^{-1+2/q} |v(r,\theta)| \le C$
and hence,
\begin{align*}
	\int_{r_1 < r < R} r^{2/q} |\nabla \omega|^2 |\omega|^{q-2}\, dx
	&\le C \int_{r>r_0} |\omega|^q \,dx, 
\end{align*}
where $C$ is a constant independent of $R$.   
The right-hand side is finite and so is the left-hand side.
Letting $R\to \infty$, we complete the proof.
\end{proof}

From the above lemma,
we obtain the following decay estimate of
$\omega$,
which is the assertion of Theorem \ref{thm1}.
\begin{lemma}\label{lem_lq3}
Under the assumption of Lemma \ref{lem_lq2}, we have
\begin{align*}
	\lim_{r\to \infty} r^{\frac{1}{q}+\frac{1}{q^2}}
		\sup_{\theta \in (0, 2\pi)} |\omega (r,\theta)| = 0.
\end{align*}
\end{lemma}
\begin{proof}
We first note that
when
$2^n > r_0$,
the inequality
\begin{align*}%
	&\int_{2^n}^{2^{n+1}} \frac{dr}{r}
		\int_0^{2\pi} |\omega|^{q-2} \left( r^2 |\omega|^2
				+ r^{1+1/q} |\omega| |\omega_{\theta}| \right)\,d\theta \\
	&\le 
		C \int_{2^n<r<2^{n+1}} |\omega|^{q-2}
			\left( |\omega|^2 + r^{1/q} |\omega||\nabla \omega| \right)\,dx \\
	&\le C \int_{r > 2^n}|\omega|^{q-2} \left( |\omega|^2 + r^{2/q} |\nabla \omega|^2 \right)\,dx
\end{align*}%
holds.
From the mean value theorem for the integration,
there exists a sequence $r_n \in (2^n, 2^{n+1})$ such that
\begin{align*}%
	&\int_{2^n}^{2^{n+1}} \frac{dr}{r}
		\int_0^{2\pi} |\omega|^{q-2}
			\left( r^2 |\omega|^2 + r^{1+1/q} |\omega| |\omega_{\theta}| \right)\,d\theta\\
	&= \log 2 \int_0^{2\pi}
		|\omega(r_n, \theta)|^{q-2}
		\left( r_n^2 |\omega(r_n,\theta)|^2
				+ r_n^{1+1/q} |\omega(r_n, \theta)| |\omega_{\theta}(r_n, \theta)| \right)\,d\theta.
\end{align*}%
Therefore, we have
\begin{align}%
\label{eq_om_rn}
	& \int_0^{2\pi}
		|\omega(r_n, \theta)|^{q-2}
		\left( r_n^2 |\omega(r_n,\theta)|^2
				+ r_n^{1+1/q} |\omega(r_n, \theta)| |\omega_{\theta}(r_n, \theta)| \right)\,d\theta \\
\notag
	&\le C \int_{r > 2^n}|\omega|^{q-2} \left( |\omega|^2 + r^{2/q} |\nabla \omega|^2 \right)\,dx.
\end{align}%
On the other hand, 
integrating the identity
\begin{align*}%
	|\omega(r_n, \theta)|^q - |\omega(r_n, \varphi)|^q
		=\int_{\varphi}^{\theta}  \frac{\partial}{\partial \theta'} |\omega(r_n, \theta')|^q \,d\theta'
\end{align*}%
with respect to
$\varphi \in [0,2\pi]$,
we have
\begin{align*}%
	2\pi |\omega(r_n, \theta)|^q - \int_0^{2\pi} |\omega(r_n, \varphi)|^q \,d\varphi
	&\le 2\pi \int_0^{2\pi} \left| \frac{\partial}{\partial \theta'} |\omega(r_n, \theta')|^q \right|\,d\theta' \\
	&\le C \int_0^{2\pi} |\omega(r_n, \theta')|^{q-1} |\omega_{\theta}(r_n, \theta')| \,d\theta'.
\end{align*}%
Multiplying it by $r_n^{1+\frac 1q}$ with $1 + \frac 1q < \frac 32$, implied by 
$q \in (2, \infty)$, we have by \eqref{eq_om_rn} that 
\begin{align*}%
	r_n^{1+1/q} |\omega(r_n, \theta)|^q
	&\le C r_n^{1+1/q} \int_0^{2\pi} |\omega(r_n, \varphi)|^q \,d\varphi \\
	&\quad + C r_n^{1+1/q} \int_0^{2\pi} |\omega(r_n, \theta')|^{q-1} |\omega_{\theta}(r_n, \theta')| \,d\theta'\\
	&\le C \int_{r > 2^n}|\omega|^{q-2} \left( |\omega|^2 + r^{2/q} |\nabla \omega|^2 \right)\,dx.
\end{align*}%
By Lemma \ref{lem_lq2}, the right-hand side tends to $0$ as $n\to \infty$.
Hence, we have
\begin{align}%
\label{eq_limrn}
	\lim_{n \to \infty} r_n^{1+1/q} \sup_{\theta \in (0, 2\pi)} |\omega(r_n, \theta)|^q = 0.
\end{align}%
By noting
$r_{n+1} < 4r_n$
and the maximum principle,
we have for
$r \in (r_n, r_{n+1})$
that
\begin{align*}%
	r^{1+1/q} \sup_{\theta \in (0, 2\pi)} | \omega(r,\theta)|^q
		&\le r_{n+1}^{1+1/q} \max\{
				\sup_{\theta \in (0, 2\pi)}| \omega(r_n,\theta)|^q,
				\sup_{\theta \in (0, 2\pi)} | \omega(r_{n+1},\theta)|^q
					\} \\
		&\le \max\{
				8r_n^{1+1/q} \sup_{\theta \in (0, 2\pi)} | \omega(r_n,\theta)|^q,
				r_{n+1}^{1+1/q} \sup_{\theta \in (0, 2\pi)} | \omega(r_{n+1},\theta)|^q
				\}.
\end{align*}%
Combining this with \eqref{eq_limrn} yields
\begin{align*}%
	\lim_{r \to \infty} r^{1+1/q} \sup_{\theta \in (0, 2\pi)} |\omega(r, \theta)|^q = 0,
\end{align*}%
which completes the proof.
\end{proof}

Next, we give the proof of Corollary \ref{cor_liouville1}.
\begin{proof}[Proof of Corollary \ref{cor_liouville1}]
We first note that,
by the Calder\'{o}n--Zygmund inequality,
$\omega \in L^q(\mathbb{R}^2)$
implies
$\nabla v \in L^q(\mathbb{R}^2)$,
and hence, we may assume
$\nabla v \in L^q(\mathbb{R}^2)$.
Then, by Theorem \ref{thm1}, we have
$\omega \to 0$
as
$|x| \to \infty$.
Since
$\omega$
satisfies the maximum principle,
we have
$\omega \equiv 0$.
\end{proof}

\subsection{Proof of the estimate of the derivative of the velocity}
Next, we prove \eqref{estnablav}.
We first prepare the following H\"{o}lder estimate of the vorticity.
\begin{lemma}\label{lem_holder}
Under the same assumption of Theorem \ref{thm1},
for any
$R > \max\{ r_0, 4\}$,
we have
\begin{align}%
\label{hol}
	| \omega(x_1) - \omega(x_2) | \le C \mu(R) |x_1 - x_2|^{1/2}
\end{align}%
for
$|x_1|, |x_2| \in (R+2, 3R-2)$ satisfying $|x_1-x_2| \le 1$,
where
$\mu(R)$
is defined by
\begin{align}%
\label{mu}
	\mu(R) = \sup_{R \le r \le 3R} |\omega(r,\theta)| \left\{ 1+ |v(r,\theta)|^{1/2} \right\}.
\end{align}%
\end{lemma}
\begin{remark}
In \cite[Lemma 6.1]{GiWe78}, the same estimate as \eqref{hol} is given with
$\mu(R) = \sup_{r \ge R} |\omega(r,\theta)| \left\{ 1+ |v(r,\theta)|^{1/2} \right\}$.
To control the velocity $v$, which may polynomially grow (see \eqref{estv}),
we slightly modify the definition of $\mu(R)$ of \cite[Lemma 6.1]{GiWe78} in the way above.
\end{remark}
\begin{proof}[Proof of Lemma \ref{lem_holder}]
The proof is the almost same as that of \cite[Lemma 6.1]{GiWe78},
however, we give a proof for reader's convenience.
Let $x_0 \in \re^2$ be $|x_0| \in (R+2, 3R-2)$ and
we consider the Dirichlet integral
\begin{align*}%
	D(r) = \int_{B_{r}(x_0)} |\nabla \omega(x)|^2 \,dx
\end{align*}%
for $r \in (0,1]$.
First, we claim that
\begin{align}%
\label{D1}
	D(1) \le C \mu(R)^2
\end{align}%
holds with an absolute constant $C>0$.
In \eqref{eq_h_en}, taking
$h(\omega) = |\omega|^2$
and replacing
$\eta_R$
by a cut-off function
$\eta \in C_0^{\infty}(B_2(x_0))$
such that
$\eta(x) = 1$ for $x \in B_1(x_0)$,
we deduce
\begin{align}%
\label{est_D1}
	D(1) &\le \int_{B_{2}(x_0)} \eta(x) |\nabla \omega(x)|^2 \,dx \\
\nonumber
		&= \frac{1}{2}
			\int_{B_2(x_0)} |\omega(x)|^2 (\Delta \eta(x) + v \cdot \nabla \eta(x) )\,dx \\
\nonumber
		&\le C \mu(R)^2.
\end{align}%

Next, we prove the Dirichlet growth condition
\begin{align}%
\label{est_Dr}
	D(r) \le C \mu(R)^2 r
\end{align}%
for $r \in (0,1)$
with an absolute constant $C>0$.
This and Morrey's lemma \cite[Theorem 7.19, Lemma 12.2]{GiTr} give the desired estimate \eqref{hol}.
To prove \eqref{est_Dr}, we put
\begin{align*}%
	\overline{\omega}(r)
	= \frac{1}{2\pi r} \int_{\partial B_r(x_0)} \omega \,ds
	= \frac{1}{2\pi} \int_0^{2\pi}
			\omega(x_{0} + r (\cos \theta, \sin \theta)) \,d\theta.
\end{align*}%
By integration by parts,
the vorticity equation
$\Delta \omega - v \cdot \nabla \omega = 0$,
and the Wirtinger inequality,
we have
\begin{align*}%
	D(r)
		&= \int_{\partial B_r(x_0)} \omega \nu \cdot \nabla \omega \,ds
			- \int_{B_r(x_0)} \omega v \cdot \nabla \omega \,dx \\
		&= \int_{\partial B_r(x_0)} \omega \nu \cdot \nabla \omega \,ds
			- \frac{1}{2} \int_{\partial B_r(x_0)} \omega ^2 \nu \cdot v \,ds \\
		&= \int_{\partial B_r(x_0)} (\omega-\overline{\omega}) \nu \cdot \nabla \omega \,ds
			+ \overline{\omega} \int_{\partial B_r(x_0)} \nu \cdot \nabla \omega \,ds
			- \frac{1}{2} \int_{\partial B_r(x_0)} \omega ^2 \nu \cdot v \,ds \\
		&\le \frac{r}{2} \int_0^{2\pi} \left[ \frac{|\omega - \overline{\omega}|^2}{r^2}
			+ |\partial_r \omega|^2 \right]_{x=x_{0} + r (\cos \theta, \sin \theta)}
				r \,d\theta
			+ \overline{\omega} \int_{B_r(x_0)} \Delta \omega \,dx
			+ C \mu(R)^2 r \\
		&\le  \frac{r}{2} \int_0^{2\pi} \left[ \frac{| \partial_{\theta} \omega |^2}{r^2}
			+ |\partial_r \omega|^2 \right]_{x=x_{0} + r (\cos \theta, \sin \theta)}
				 r \,d\theta
			+ \overline{\omega} \int_{B_r(x_0)} \Delta \omega \,dx
			+ C \mu(R)^2 r
			\\
		&\le \frac{r}{2} \int_0^{2\pi} |(\nabla \omega)(x_0 + r(\cos \theta, \sin \theta))|^2 r \,d\theta
			+ \overline{\omega} \int_{B_r(x_0)} v \cdot \omega \,dx
			+ C \mu(R)^2 r \\
		&\le \frac{r}{2} D'(r) + C\mu(R)^2 r,
\end{align*}%
where
$\nu$
is the unit outward normal vector of $\partial B_r(x_0)$.
We rewrite the above as
\begin{align*}%
	\frac{d}{dr} \left( \frac{D(r)}{r^2} \right) \ge - C \mu(R)^2 r^{-2},
\end{align*}%
and integrate it over $[r,1]$, and use \eqref{est_D1} to obtain
\begin{align*}%
	\frac{D(r)}{r^2} \le D(1) + C\mu(R)^2 \left(\frac{1}{r}-1 \right) \le C\mu(R)^2 \frac{1}{r},
\end{align*}%
which gives the Dirichlet growth condition \eqref{est_Dr}.
\end{proof}
\begin{proof}[Proof of \eqref{estnablav}]
To prove \eqref{estnablav},
we follow the argument of \cite[Theorem 7]{GiWe78} and
identify $x=(x_1, x_2) \in \mathbb{R}^2$ and
$z = x_1+ ix_2 \in \mathbb{C}$.
We put
$f(z) = v_1(z) - i v_2(z)$
and
$\partial_z = \frac{1}{2}( \partial_{x_1} - i \partial_{x_2})$,
$\partial_{\bar{z}} = \frac{1}{2}(\partial_{x_1} + i \partial_{x_2})$.
Then, we easily obtain
\begin{align*}%
	\partial_{\bar{z}} f(z) = -\frac{i}{2} \omega(z)
\end{align*}%
and
$|\partial_z f(z)|^2 + |\partial_{\bar{z}}f(z)|^2 = \frac{1}{2}|\nabla v(x)|^2$.
Therefore, by noting \eqref{estomegav},
it suffices to show
\begin{align}%
\label{fz}
	|\partial_z f(z)| = o ( |z|^{-\frac{1}{q}-\frac{1}{q^2}} \log |z| ) \quad ( |z| \to \infty).
\end{align}%
Let
$z_0 \in \mathbb{C}$
satisfy
$|z_0| > \max\{ r_0, 2 \}$
and let
$R= |z_0|/2$. 
The Cauchy integral formula \cite[Theorem 1.2.1]{Hor} implies
\begin{align*}%
	f(z) &= \frac{1}{2\pi i} \left\{ \int_{\partial B_R(z_0)} \frac{f(\zeta)}{\zeta - z} \,d\zeta
						+ \int_{B_R(z_0)} \frac{\partial_{\bar{z}}f(\zeta)}{\zeta - z}
							\,d\zeta \wedge d\bar{\zeta} \right\} \\
	&=  \frac{1}{2\pi i} \left\{ \int_{\partial B_R(z_0)} \frac{f(\zeta)}{\zeta - z} \,d\zeta
			+ \int_{B_R(z_0)} \frac{\omega(\zeta)}{\zeta - z} \,d\xi d\eta \right\} \\
	&=  \frac{1}{2\pi i} \left\{ 
			\int_{\partial B_R(z_0)} \frac{f(\zeta)}{\zeta - z} \,d\zeta
			+ \int_{B_R(z_0)} \frac{\omega(\zeta) - \omega(z_0)}{\zeta - z} \,d\xi d\eta
			+ \omega(z_0) \int_{B_R(z_0)} \frac{d\xi d\eta}{\zeta - z}
		\right\}
\end{align*}%
for $z \in B_R(z_0)$,
where we denote
$\zeta = \xi + i \eta$
and use
$d \zeta \wedge d \bar{\zeta} = -2 i d\xi \wedge d\eta$.
By the Cauchy integral formula again,
the last term is computed as
\begin{align*}%
	\int_{B_R(z_0)} \frac{d\xi d\eta}{\zeta - z}
	&= \frac{i}{2} \int_{B_R(z_0)} \frac{\partial_{\bar{\zeta}}(\bar{\zeta} - \bar{z_0})}{\zeta - z}
		\,d\zeta \wedge d\bar{\zeta} \\
	&= -\pi (\bar{z} - \bar{z_0} ) 
		- \frac{i}{2} \int_{\partial B_R(z_0)} 
			\frac{\bar{\zeta} - \bar{z_0}}{\zeta - z} \,d\zeta.
\end{align*}%
Here, we claim that the last term of the right-hand side vanishes,
because letting
$\zeta = z_0 + Re^{i\theta}$ and $w = (z-z_0)/R$, we compute
\begin{align*}
	\int_{\partial B_R(z_0)} 
			\frac{\bar{\zeta} - \bar{z_0}}{\zeta - z} \,d\zeta
	= iR \int_0^{2\pi} \frac{d\theta}{e^{i\theta} - w} =: F(w)
\end{align*}
and we easily see that
$F(w)$ is a holomorphic function of $w$ on $|w|<1$
with $\partial_w^n F(0) = 0$
for all $n \in \mathbb{N}\cup \{0\}$, that is, $F(w) \equiv 0$ for $|w|<1$.
Hence, the above expression of $f(z)$ is differentiable at $z=z_0$
and we have
\begin{align*}%
	\partial_z f(z_0)
		&= \frac{1}{2\pi i} \left\{ 
			\int_{\partial B_R(z_0)} \frac{f(\zeta)}{(\zeta - z_0)^2} \,d\zeta
			+ \int_{B_R(z_0)} \frac{\omega(\zeta) - \omega(z_0)}{(\zeta - z_0)^2} \,d\xi d\eta
		\right\}.
\end{align*}%
By using \eqref{estv}, the first term of the right-hand side is estimated as
\begin{align*}%
	\left| \int_{\partial B_R(z_0)} \frac{f(\zeta)}{(\zeta - z_0)^2} \,d\zeta \right|
		&\le C R^{-2/q}.
\end{align*}%
By Lemma \ref{lem_holder}, for $\delta \in (0,1)$, the second term is estimated as
\begin{align*}%
	&\left| \int_{B_R(z_0)} \frac{\omega(\zeta) - \omega(z_0)}{(\zeta - z_0)^2} \,d\xi d\eta \right| \\
		&\le
			\left( \int_{B_1(z_0)} + \int_{B_R(z_0)\setminus B_1(z_0)} \right)
				\frac{|\omega(\zeta) - \omega(z_0)|}{|\zeta - z_0|^2} \,d\xi d\eta \\
		&\le C \mu(R) \int_0^{2\pi} \int_0^{\delta} \frac{\rho^{1/2}}{\rho^2} \rho \,d\rho d\theta
			+ C \sup_{|\zeta| > R} |\omega(\zeta)|
				\int_0^{2\pi} \int_{\delta}^R \frac{1}{\rho^2} \rho \,d\rho d\theta \\
		&\le C ( \mu(R) \delta^{1/2}
			+ \log \frac{R}{\delta} \sup_{|\zeta| > R} |\omega(\zeta)| ).
\end{align*}%
Finally, applying \eqref{estv} and \eqref{estomegav} to $\mu(R)$,
and optimizing the above estimate by choosing $\delta = R^{-1+2/q}$,
we conclude
\begin{align*}%
		&\mu(R) \delta^{1/2} 
			+ \log \frac{R}{\delta} \sup_{|\zeta| > R} |\omega(\zeta)| \\
		& = o(R^{-\frac{1}{q} - \frac{1}{q^2} + \frac{1}{2}\left(1-\frac{2}{q} \right)}
					\cdot R^{-\frac{1}{2} \left(1-\frac{2}{q} \right)} )
			+ o( (\log R) \cdot R^{-\frac{1}{q} - \frac{1}{q^2}} ) \\
		& = o(R^{-\frac{1}{q} - \frac{1}{q^2}} \log R)
		\quad (R\to \infty),
\end{align*}%
which gives \eqref{fz}.
\end{proof}

\section{Proof of Theorem \ref{thm_energy_id} and Corollary \ref{cor_liouville}}
In this section, we give the proof of Theorem \ref{thm_energy_id} and Corollary \ref{cor_liouville}.
The idea of the proof is using the truncating function
$h(\omega)$
to control the nonlinear term
under the assumption
$\omega \in L^q( \mathbb{R}^2 \times I )$.

\begin{proof}[Proof of Theorem \ref{thm_energy_id}]
Similarly to the proof of Theorem \ref{thm1},
let 
$\omega_* > 0$
be a constant, and define a function
$h(\omega)$
by
\begin{align*}
	h(\omega)
	= \begin{cases}
		|\omega|^q &(|\omega| \le \omega_*),\\
		\omega_*^{q-1} ( q|\omega| - (q-1) \omega_* )&(|\omega|\ge \omega_*).
	\end{cases}
\end{align*}
We also take a function
$\psi \in C_0^{\infty}(\mathbb{R}^2)$
so that
\begin{align*}%
	\psi (x) = \begin{cases} 1 &(|x|\le 1),\\ 0&(|x| \ge 2) \end{cases}
\end{align*}%
holds, and for
$R>0$,
we define
$\psi_R(x) = \psi(\frac{x}{R})$.
Then, noting
$\diver \omega (x,t) = 0$,
we have the following identity in the same way to the proof of Lemma \ref{lem_lq2}:
\begin{align*}%
	&\diver\left[
		\psi_R \nabla h - h \nabla \psi_R - \psi_R h v \right] \\
		&= \psi_R h'' |\nabla \omega|^2
			-h \left[ \Delta \psi_R + v \cdot \nabla \psi_R \right]
			+ \psi_R h' \left[ \Delta \omega - v \cdot \nabla \omega \right] \\
		&= \psi_R h'' |\nabla \omega|^2
			-h \left[ \Delta \psi_R + v \cdot \nabla \psi_R \right]
			+ \psi_R h' \omega_t.
\end{align*}%
Using
$\psi_R h' \omega_t = \frac{d}{dt} [\psi_R h(\omega) ]$
and integrating the above identity over
$\mathbb{R}^2$,
we deduce
\begin{align*}%
	\frac{d}{dt} \int_{\mathbb{R}^2} \psi_R h(\omega) \,dx
		+ \int_{\mathbb{R}^2} \psi_R h''(\omega) |\nabla \omega|^2 \,dx
	&= \int_{\mathbb{R}^2} h(\omega) (\Delta \psi_R + v \cdot \nabla \psi_R) \,dx.
\end{align*}%
Moreover, integrating it over 
$[t',t]$
yields
\begin{align*}%
	&\int_{\mathbb{R}^2} \psi_R h(\omega(x,t)) \,dx
		+ \int_{t'}^t \int_{\mathbb{R}^2}
				\psi_R h''(\omega(x,\tau)) |\nabla \omega(x,\tau)|^2 \,dx d\tau \\
	&=  \int_{\mathbb{R}^2} \psi_R h(\omega(x,t')) \,dx
		+ \int_{t'}^t \int_{\mathbb{R}^2}
				h(\omega(x,\tau)) (\Delta \psi_R + v(x,\tau) \cdot \nabla \psi_R) \,dx d\tau.
\end{align*}%
From
$h''(\omega) = q(q-1) |\omega|^{q-2} \ (|\omega| < \omega_*)$
and
$h''(\omega) = 0 \ (|\omega| > \omega_*)$,
we further obtain
\begin{align}%
\label{eq_om1}
	&\int_{\mathbb{R}^2} \psi_R h(\omega(x,t)) \,dx
		+ q(q-1) \int_{t'}^t
			\int_{\{ x\in \re^2 ; |\omega(x,\tau)|<\omega_* \}}
				\psi_R |\omega(x,\tau)|^{q-2} |\nabla \omega(x,\tau)|^2 \,dx d\tau \\
\notag
	&=   \int_{\mathbb{R}^2} \psi_R h(\omega(x,t')) \,dx
		+ \int_{t'} ^t \int_{B_{2R}\setminus B_R}
			h(\omega(x,\tau)) \left( \Delta \psi_R + v(x,\tau) \cdot \nabla \psi_R \right) \,dx d\tau.
\end{align}%
Here,
$B_{R}$
denotes the open disc centered at the origin with radius $R$.
Now, we estimate the right-hand side.
By an interpolation argument, we first note that
$h(\omega) \le C \omega_*^{q-s} |\omega|^s \quad (1\le s \le q)$
holds.
From this and
$\omega \in  L^q(\mathbb{R}^2\times I)$,
we obtain
\begin{align*}%
	\int_{t'}^t \int_{B_{2R}\setminus B_R} h(\omega) |\Delta\psi_R| \,dxd\tau
	&\le C R^{-2} \int_{t'}^t \| \omega (\cdot, \tau) \|_{L^q(B_{2R}\setminus B_R)}^q \,d\tau
	\to 0 \quad (R\to \infty).
\end{align*}%
Next, we estimate
\begin{align*}%
	\int_{t'}^t \int_{B_{2R}\setminus B_R} h(\omega) | v \cdot \nabla \psi_R |\,dxd\tau
	&\le C \omega_* \int_{t'}^t \int_{B_{2R} \setminus B_R}
			|\omega|^{q-1} |\nabla \psi_R| |v-v(0,\tau)| \,dxd\tau\\
	&\quad + \int_{t'}^t \int_{B_{2R}\setminus B_R}
				|\omega|^q |\nabla \psi_R | |v(0,\tau)| \,dxd\tau \\
	&=: J_1 + J_2
\end{align*}%
and easily see
\begin{align*}%
	J_2
	&\le C R^{-1} \left( \sup_{t' < \tau < t} |v(0,\tau)| \right)
	\int_{t'}^t \| \omega(\cdot,\tau) \|_{L^q(B_{2R}\setminus B_R)}^q \,d\tau
		\to 0 \quad (R\to \infty).
\end{align*}%
For the term $J_1$,
we note that, by the Calder\'{o}n--Zygmund inequality,
$\omega (\cdot, t) \in L^q(\mathbb{R}^2)$ (a.e. $t \in I$)
implies
$\nabla v (\cdot, t) \in L^q(\mathbb{R}^2)$ (a.e. $t \in I$),
and the Sobolev embedding leads to
$v(\cdot, t) \in \dot{C}^{0,1-2/q}(\mathbb{R}^2)$ (a.e. $t \in I$).
Namely, we have
\begin{align*}%
	|v(x,\tau) - v(0, \tau) | \le C |x|^{1-\frac{2}{q}} \| \omega (\cdot, \tau) \|_{L^q(\mathbb{R}^2)}
	\quad ({\rm a.e.}\, \tau \in (t',t)).
\end{align*}%
Using this, we estimate
\begin{align*}%
	J_1 &\le C\omega_* R^{-1}
		\int_{t'}^t
			\left( \int_{B_{2R}\setminus B_R} |\omega (x,\tau)|^q \,dx \right)^{\frac{1}{q'}}
			\left( \int_{B_{2R}\setminus B_R} | v(x,\tau) - v(0,\tau) |^q \,dx \right)^{\frac{1}{q}} \,d\tau\\
	&\le C \omega_* R^{-1} \int_{t'}^t
		\| \omega (\cdot, \tau) \|_{L^q(B_{2R}\setminus B_R)}^{\frac{q}{q'}}
			\left( \int_{B_{2R}\setminus B_R} \left| (2R)^{1-\frac{2}{q}} \| \omega(\cdot,\tau) \|_{L^q(\mathbb{R}^2)}
				\right|^q \,dx \right)^{\frac{1}{q}} \,d\tau \\
	&\le C \omega_* \int_{t'}^t
		\| \omega (\cdot, \tau) \|_{L^q(B_{2R}\setminus B_R)}^{q-1}
		\| \omega(\cdot,\tau) \|_{L^q(\mathbb{R}^2)} \,d\tau.
\end{align*}%
The right-hand side tends to
$0$
as
$R\to \infty$
due to
$\omega \in L^q(\mathbb{R}^2\times I)$
and the Lebesgue dominated convergence theorem.
Therefore, letting $R\to \infty$ in \eqref{eq_om1}, we conclude
\begin{align*}%
	&\int_{\mathbb{R}^2} h(\omega(x,t)) \,dx
		+ q(q-1) \int_{t'}^t \int_{\{ x\in \re^2 ; |\omega(x,\tau)|<\omega_* \}}
			|\omega|^{q-2} |\nabla \omega|^2 \,dxd\tau \\
		&= \int_{\mathbb{R}^2} h(\omega(x,t'))\,dx.
\end{align*}%
Finally,
by the monotone convergence theorem,
taking the limit
$\omega_* \to \infty$
implies
\begin{align*}%
	&\int_{\mathbb{R}^2} |\omega(x,t)|^q \,dx
	+ q(q-1) \int_{t'}^t \int_{\mathbb{R}^2}
		|\omega(x,\tau)|^{q-2} |\nabla \omega(x,\tau)|^2 \,dxd\tau \\
	&= \int_{\mathbb{R}^2} |\omega(x,t')|^q \,dx,
\end{align*}%
which completes the proof of Theorem \ref{thm_energy_id}.
\end{proof}
Corollary \ref{cor_liouville} (i) is easily obtained by
the $L^q$-energy identity in Theorem \ref{thm_energy_id},
the condition
$\omega(x,0) = 0$
implies
\begin{align*}%
	\int_{\mathbb{R}^2} |\omega(x,t)|^q \,dx = 0
\end{align*}%
for any
$t\in [0,T]$,
namely,
$\omega \equiv 0$.

For Corollary \ref{cor_liouville} (ii),
by Theorem \ref{thm_energy_id},
for any $t' < t < 0$ we have
\begin{align}
\label{om_lq}
	\| \omega(t) \|_{L^q(\re^2)} \le \| \omega (t') \|_{L^q(\re^2)}.
\end{align}
Since $\omega \in L^q(\re^2 \times I)$,
we have
$\lim_{n\to \infty} \| \omega(t_n) \|_{L^q(\re^2)} = 0$
along with some sequence
$\{ t_n \}_{n=1}^{\infty}$ satisfying $t_n \to - \infty$ as $n\to \infty$.
Thus, taking $t' = t_n$ in \eqref{om_lq} and letting $n\to \infty$,
we conclude
$\| \omega (t) \|_{L^q(\re^2)} = 0$
for any $t \in I$.

\appendix
\section{Alternate proof of Corollary \ref{cor_liouville1}}
In the appendix, we give an alternate proof of Corollary \ref{cor_liouville1}.
Applying the integration by parts to the vorticity equation
\begin{align*}%
	-\Delta \omega + v \cdot \nabla \omega = 0,
\end{align*}%
we have
\begin{align}%
\label{vor_ene_id}
	\int_{\re^2} \nabla \omega \cdot \nabla \varphi \,dx
		+ \int_{\re^2} v \cdot \nabla \omega \varphi \,dx = 0
\end{align}%
for any
$\varphi \in C_0^1(\re^2)$.
Let
$\psi \in C_0^{\infty}(\re^2)$
be a test function satisfying
\begin{align*}%
	\psi (x) = \begin{cases}
		1,&|x|\le 1,\\
		0,&|x| \ge 2
		\end{cases}
\end{align*}%
and $0 \le \psi \le 1$.
Let $R>0$ be a parameter and we define
$\psi_R(x) = \psi (x/R)$.
Taking
$\varphi(x) = |\omega|^{q-2} \omega \psi_R$
in \eqref{vor_ene_id}, we have
\begin{align}%
\label{vor_ene_id2}
	(q-1) \int_{\re^2} |\nabla \omega|^2 |\omega|^{q-2} \psi_R \,dx
	&=
		- \int_{\re^2} |\omega|^{q-2} \omega \nabla \omega \cdot \nabla (\psi_R) \,dx \\
\notag
	&\quad 
		- \int_{\re^2} (v \cdot \nabla \omega) |\omega|^{q-2} \omega \psi_R\,dx\\
\notag
	&=: I + I\!I.
\end{align}%
We estimate $I$ as
\begin{align}%
\label{I}
	I &= - \frac{1}{q} \int_{\re^2} \nabla ( |\omega|^{q}) \cdot \nabla (\psi_R) \,dx\\
\notag
	&= \frac{1}{q} \int_{\re^2} | \omega |^{q} \Delta (\psi_R) \,dx\\
\notag
	&\le C(q) R^{-2} \int_{B_R\setminus B_{R/2}} | \omega |^{q} \,dx\\
\notag
	&= C(q) R^{-2} \| \omega \|_{L^q(B_R\setminus B_{R/2})}^q.
\end{align}%
Next, we estimate $I\!I$.
When
$\omega \in L^q(\re^2)$,
by using the Calder\'{o}n--Zygmund inequality,
we may assume
$\nabla v \in L^q(\re^2)$.
This and $q>2$ implies that
$v \in \dot{C}^{0,1-2/q}(\re^2)$ and
\begin{align}%
\label{hol_conti}
	| v(x) - v(y) | \le C(q) \| \omega \|_{L^q(\re^2)} | x - y|^{1-2/q}.
\end{align}%
In view of this, we compute
\begin{align}%
\label{II}
	I\!I &=
		- \int_{\re^2} (v \cdot \nabla \omega) |\omega|^{q-2} \omega \psi_R\,dx \\
\notag
		&= - \frac{1}{q} \int_{\re^2} v \cdot \nabla ( | \omega |^{q} ) \psi_R \,dx\\
\notag
		&= \frac{1}{q} \int_{\re^2} v \cdot \nabla (\psi_R) | \omega |^{q} \,dx\\
\notag
		&= \frac{1}{q} \int_{\re^2} ( v(x) - v(0) ) \cdot \nabla (\psi_R) | \omega |^{q} \,dx\\
\notag
		&\quad + \frac{1}{q} \int_{\re^2} v(0) \cdot \nabla (\psi_R) | \omega |^{q} \,dx\\
\notag
		&=: I\!I\!I + IV.
\end{align}%
By \eqref{hol_conti}, $I\!I\!I$ is estimated as
\begin{align}%
\label{III}
	I\!I\!I &\le C(q) \| \omega \|_{L^q(\re^2)}
		\int_{B_R\setminus B_{R/2}} |x|^{1-2/q} | \omega |^{q} | \nabla \psi_R | \,dx\\
\notag
	&\le C(q) R^{-2/q} \| \omega \|_{L^q(\re^2)} \| \omega \|_{L^q(B_R\setminus B_{R/2})}^q.
\end{align}%
We also have
\begin{align}%
\label{IV}
	IV &\le C(q,v(0)) \int_{B_R\setminus B_{R/2}} | \omega |^{q} | \nabla \psi_R | \,dx\\
\notag
	&\le C(q,v(0)) R^{-1} \| \omega \|_{L^q(B_R\setminus B_{R/2})}^q.
\end{align}%
Therefore, combining \eqref{vor_ene_id2}, \eqref{I}, \eqref{III} and \eqref{IV},
we conclude
\begin{align*}%
	\int_{\re^2} |\nabla \omega|^2 |\omega|^{q-2} \psi_R \,dx
	&\le C(q,l,v(0)) (R^{-2} + R^{-2/q} \| \omega \|_{L^q(\re^2)} + R^{-1} )
		\| \omega \|_{L^q(B_R\setminus B_{R/2})}^q.
\end{align*}%
Letting $R\to \infty$, we have
\begin{align*}%
	\int_{\re^2} |\nabla \omega|^2 |\omega|^{q-2} \,dx = 0,
\end{align*}%
which leads to
$|\nabla \omega| | \omega| = 0$ in $\re^2$,
namely,
$\nabla (\omega^2) = 0$
and hence,
$\omega$
must be a constant.
This and
$\omega \in L^q(\re^2)$ with some $q \in (2,\infty)$ show $\omega \equiv 0$.
From this and $\diver v = 0$,
it follows that
$v$
is a harmonic vector function,
and we conclude the constancy of
$v$
by the Liouville theorem for the harmonic functions.

\section*{acknowledgement}
The authors would like to thank Professor Yasunori Maekawa for helpful discussions.
This work was supported by JSPS Grant-in-Aid for Scientific Research(S)
Grant Number JP16H06339.

\end{document}